\documentclass[12pt]{amsart}
\usepackage[margin=1.2 in]{geometry}

\usepackage{graphicx} 
\usepackage{amsmath,amsthm,amssymb,amsfonts}
\usepackage{mathtools}
\usepackage{physics}
\usepackage{hyperref}
\usepackage[dvipsnames]{xcolor}
\usepackage{tikz-cd}
\usepackage{comment}
\usepackage[doi=false,isbn=false,url=false,eprint=false]{biblatex}
\newtheorem{theorem}{Theorem}[section]

\newtheorem{lemma}[theorem]{Lemma}
\newtheorem{Corollary}[theorem]{Corollary}
\theoremstyle{definition}

\newtheorem*{remark}{Remark}

\DeclareMathOperator{\R}{\mathbb{R}}

\DeclareMathOperator{\Z}{\mathbb{Z}}
\DeclareMathOperator{\N}{\mathbb{N}}

\newcommand{\deriv}[2][]{\frac{d #1}{d #2}}

\newcommand{\nor}{\mathsf{N}}
\newcommand{\tang}{\mathsf{T}}

\newcommand{\ttime}{\tilde{t}}

\newcommand{\tkappa}{\tilde{\kappa}}
\newcommand{\tsigma}{\tilde{\sigma}}

\renewcommand{\l}{\lambda}
\newcommand{\ovth}{\overline\theta}
\newcommand{\unth}{\underline\theta}
\newcommand{\osigma}{\overline\sigma}

\title[Stability of the Shrinking Semi-Circle]{Stability of the Shrinking Semi-Circle Under the Free Boundary Curve Shortening Flow}

\author{Theodora Bourni}\address{Department of Mathematics, University of Tennessee, Knoxville, TN, 37996, USA}\definecolor{armygreen}{rgb}{0.36, 0.54, 0.66}
\author{Nathan Burns}\address{Department of Mathematics, University of Tennessee, Knoxville, TN, 37996, USA}\definecolor{armygreen}{rgb}{0.36, 0.54, 0.66}
\author{Mat Langford}\address{Mathematical Sciences Institute, Australian National University, Canberra, ACT, 0200, Australia}\definecolor{airforceblue}{rgb}{0.36, 0.54, 0.66}

\begin{document}

\begin{abstract} We establish a sharp rate of convergence for a free-boundary curve shortening flow in a convex domain in $\R^2$ which converges in finite time to a round half-point.
\end{abstract}

\maketitle


\section{Introduction}
Mean curvature type flows provide a fundamental framework for understanding the geometric and analytic properties of submanifolds through curvature driven evolution. Among these, the curve shortening flow has played a central role as a model problem and  understanding the long-time behavior and singularity formation of solutions to curve shortening flow has led to deep insights and has served as a testing ground for techniques later extended to higher-dimensional mean curvature flow. In this work, we focus on the free-boundary curve shortening flow in a convex domain of $\R^2$,
with particular emphasis on its long-time behavior. 

In the case of convex embedded curves without boundary, Gage and Hamilton showed that the flow contracts the curve to a round point with exponential convergence \cite{GaHa86}. Embeddedness is essential: Angenent exhibited examples of plane curves developing cusps under curve shortening flow \cite{Angenent1991}. Convexity, however, is not necessary, as Grayson proved that any simple closed curve becomes convex under the flow \cite{Grayson}.
Analogous results hold in higher dimensions. For compact, convex hypersurfaces without boundary, Huisken showed that the hypersurfaces shrink to a round point in finite time under mean curvature flow \cite{Hooken2}.

When boundaries are present, the analysis becomes more subtle due to the interaction between curvature effects and boundary conditions. The problem was first addressed by Huisken in the graphical setting \cite{Hooken3}, and later in greater generality by Stahl \cite{Stahl2}, who proved that a convex hypersurface with free boundary on a convex domain converges to a  point. In dimensions greater than two, Stahl also proved roundness of the limiting profile. 
For planar curves, however, roundness of the limit remained open until recently. This question was resolved in \cite{fbcsf}, where it was shown that any compact, embedded free-boundary curve shortening flow in a convex planar domain of class $C^2$
 either converges in infinite time to a critical chord or shrinks in finite time to a round half-point \cite{fbcsf}. This result can be viewed as the free-boundary analogue of Grayson's theorem \cite{Grayson}, and it also improves upon Stahl's work, even in the convex case, by establishing convergence to a \emph{round} half-point.

However, the blow-up analysis in \cite{fbcsf} does not provide a rate of convergence to the shrinking semicircle after rescaling about the extinction point. In this work, we address this gap by establishing a sharp rate of convergence to the shrinking semicircle via a stability analysis. This yields a level of quantitative precision comparable to the classical closed-curve case treated in \cite{GaHa86}.
Quantitative estimates and a precise geometric understanding of the limiting behavior play a crucial role in uniqueness and classification results. For instance, \cite{Wang2002} uses the forward asymptotics of Gage–Hamilton to prove uniqueness of the shrinking circle among convex entire ancient solutions. Similarly, \cite{ADS2019, BourniLangfordTinaglia2021} establish sharp backward-in-time asymptotics to deduce rotational symmetry of the solutions.

Our main result shows that, after appropriate rescaling about the extinction point, solutions converge to the shrinking semicircle with an explicit and sharp rate in all $C^k$ norms.

\begin{theorem}
	Let \(\Omega \subset \R^{2}\) be a convex domain with \(C^{2}\) boundary and let \(\{\Gamma_{t}\}_{t \in [0,T)}\) be a maximal free boundary curve shortening flow starting from a properly embedded curve \(\Gamma_{0}\) inside \(\Omega\). If \(T < \infty\), then \(\Gamma_{t}\) converges uniformly to some \(p \in \partial\Omega\), and, denoting by  \(\mathcal R\) the rotation of \(\R^{2}\) which maps the outward normal to \(\partial\Omega\) at \(p\) to \(-e_{2}\), the rescaled curves 
	\[
	\tilde{\Gamma}_{t} \coloneqq \frac{\mathcal R(\Gamma_{t} - p)}{\sqrt{2(T-t)}}
	\]
	converge as \(t \to T\) to the unit semi-circle in \(\R^{2}_{+}\) uniformly in the \(C^{k}\)-topology, for every \(k \geq 0\), at a rate no slower than \((T-t)^{1-\delta}\), for any \(\delta > 0\).
\end{theorem}

The proof is based on a stability analysis of the shrinking semicircle as a self-similar solution of the rescaled flow. The linearized operator possesses unstable modes corresponding to time translation and horizontal translation, which obstruct direct convergence estimates. To overcome this difficulty, we introduce a normalization of the flow that dynamically modulates these symmetries by fixing two geometric quantities: the enclosed area between the curve and the support curve, and a suitably defined center of mass. With this choice of normalization, the projections onto the unstable modes can be perturbed away, allowing us to close the stability estimates and obtain sharp convergence rates.

Our approach is inspired by Andrews’ treatment of the affine normal flow for closed hypersurfaces \cite[\S16.5.2]{MR4249616}, adapted here to the free-boundary setting and to the specific structure of the curve shortening flow. The presence of the boundary however creates a lot more difficulties. A significant problem is that using the Gauss map parametrization the domain of definition is not constant in time. This causes the stability analysis after linearizing around a shrinking semicircle far from trivial.

We remark that, in the special case where the domain $\Omega$ is a hyperplane, the decay rate given in the main theorem can be improved to $2-\delta$ (instead of $1-\delta$). This follows from the proof of the theorem with minor modifications. Of course, this case can also be derived from the Gage–Hamilton theorem \cite{GaHa86}, after reflecting the solution.

\section{Preliminaries}

Consider a compact solution $\{\Gamma_t\}_{t\in[0,T)}$ to the free boundary curve shortening flow in a convex, locally uniformly convex domain \(\Omega \subset \R^{2}\). Denoting the boundary of this domain by \(\Sigma \coloneqq \partial \Omega\), suppose that \(\Gamma_t\) converges to the round half-point \(p \in \Sigma\) as \(t \to T < \infty\); i.e., 
\begin{equation}\label{1-conv}
	\frac{\Gamma_t - p}{\sqrt{2(T-t)}} \overset{t \to T}{\to} \mathcal C ^+ 
\end{equation}
in the smooth topology, where \(\mathcal C^+\) is a unit semi-circle in \(\overline{H}\), the closed half-plane which supports \(\Omega - p\) at the origin. Up to a rigid motion, a time-translation and a parabolic rescaling, we may arrange so that \(p\) is the origin, that \(\overline{H}\) is the upper half-plane and that the initial curve \(\Gamma_{0} \) is \(\varepsilon\)-close in the smooth topology to the unit semi-circle in $\overline H$, for any \(\varepsilon > 0\) we like. 
In particular, \(\Gamma_{t}\) will be convex for \(t \geq 0\) and therefore we may consider the boundary value problem for its support function \(\sigma(\cdot,t)\). Adopting a Gauss map parametrisation, that is parametrizing by  the angle \(\theta \in [\underline{\theta}(t),\overline{\theta}(t)]\) the outward unit normal \(\nor = (\cos\theta, \sin\theta)\) makes with the positive \(x\)-axis, then \(\sigma\) satisfies the following problem:
\begin{equation}\label{eq:support_pde}
	\begin{cases}
		\sigma_{t}(\theta, t) &= -(\sigma_{\theta\theta} + \sigma)^{-1}\, ,\,\, (\theta, t) \in (\underline{\theta}, \overline{\theta})\times[0, T)\,, \\
		\sigma_{\theta}(\underline{\theta}, t) &= -\sigma_{\Sigma}(\overline{\theta}_{\Sigma})\,,\,\,t\in [0, T)\,,\\
		\sigma_{\theta}(\overline{\theta}, t) &= \sigma_{\Sigma}(\underline{\theta}_{\Sigma})\,,\,\,t\in [0, T)\,,
	\end{cases}
\end{equation}
where \(\sigma_{\Sigma}\) is the support function of \(\Sigma\) parametrised by the Gauss map so that 
\begin{equation}\label{thetaS}
	(\underline{\theta}_{\Sigma}(t), \overline{\theta}_{\Sigma}(t)) = (\tfrac{\pi}{2} + \overline{\theta}(t), \tfrac{3\pi}{2} + \underline{\theta}(t)).
\end{equation}
Note that, for notational convenience, we will often omit the argument $t$ in the endpoints of the angle intervals $\overline \theta, \underline{\theta}, \overline{\theta}_{\Sigma}, \underline{\theta}_{\Sigma}$, as we did for example in \eqref{eq:support_pde}.

By the smooth convergence in \eqref{1-conv}, there exists a constant $C$, such that the curvature $\kappa$ of the evolving curves satisfies
\begin{equation}\label{eq:curvature_estimates1}
	C^{-1} \leq \sqrt{2(T-t)} \kappa(\theta, t) \leq C\,,\,\, \forall (\theta, t)\in [\underline{\theta}, \overline{\theta}]\times[0, T) \,.
\end{equation}
Since 
\begin{equation}\label{thetaev}
\deriv{t}{\overline{\theta}} = \kappa_{\Sigma}(\underline\theta_\Sigma)\kappa(\overline\theta,t)\;\;\text{and}\;\; \deriv{t}{\underline{\theta}} = -\kappa_{\Sigma}(\overline\theta_\Sigma)\kappa(\underline\theta,t)\,,
\end{equation}
where \(\kappa_{\Sigma}\) is the curvature of the boundary $\Sigma$, we obtain
\begin{equation}\label{eq:motion_of_boundary_angles}
	\frac{C^{-1}}{\sqrt{2(T - t)}} \leq \deriv{t}{(\overline{\theta} - \underline{\theta})} \leq \frac{C}{\sqrt{2(T-t)}}.
\end{equation}
Integrating this estimate from time \(t\) to \(T\) yields, 
\begin{equation}\label{eq:difference_in_angle}
	C^{-1}\sqrt{2(T-t)} \leq \pi - (\overline{\theta} - \underline{\theta}) \leq C\sqrt{2(T-t)}.
\end{equation}

\subsection{Linear stability}\label{sub-lenearstability}

The linearisation of \eqref{eq:support_pde} at the shrinking semi-circle  is given by
\[
	\begin{cases}
		v_{t}(\theta, t) &= \frac{-1}{2t}(v_{\theta\theta}(\theta, t) + v(\theta, t))\,,\,\, \theta \in (0,\pi)\,,\\
		v_{\theta}(0, t) &= 0\,, \\
        v_{\theta}(\pi, t) &= 0\,,
	\end{cases}
\]
where $t\in (-\infty, 0)$.
Introducing the rescaled variables
\[
	\tilde{v}(\theta,\tilde t) := e^{\tilde t}v(\theta,t)\,,\;\; \tilde t = -\frac{1}{2}\log(-2t)
\]
transforms this equation into
\[
	\begin{cases}
		\tilde v_{\tilde t}(\theta, \tilde t) &= \tilde v_{\theta\theta}(\theta, \tilde t) + 2\tilde v(\theta, \tilde t)\,,\,\, \theta \in (0,\pi)\,,\\
		\tilde v_{\theta}(0, \tilde t) &= 0\,, \\
        \tilde v_{\theta}(\pi, \tilde t) &= 0\,,
	\end{cases}
\]	
where $\tilde t\in (-\infty, +\infty)$ here. Separation of variables then gives the representation 
\[
	\tilde v(\theta,\tilde t ) = \sum_{j = 0}^{\infty}C_{j}e^{(2-j^{2})\tilde t }\cos(j\theta).
\]
Note that the unstable mode \(j = 0\) corresponds to time-translations and the unstable mode \(j = 1\) corresponds to horizontal translation. Indeed, denoting by $\gamma(\cdot, t)$ the Gauss map parametrization of $\Gamma_t$, if \(\gamma_{\varepsilon}(\theta,t) \coloneqq \gamma(\theta,t - \varepsilon)\), then \(\tfrac{d}{d\varepsilon}\big|_{\varepsilon = 0}\ev{\gamma_{\varepsilon}, \nor} = 1\), and if \(\gamma_{\varepsilon}(\theta,t) \coloneqq \gamma(\theta,t) + \varepsilon e_{1}\), then \(\tfrac{d}{d\varepsilon}\big|_{\varepsilon = 0}\ev{\gamma_\varepsilon, \nor} = \cos\theta\).

\subsection{The enclosed area functional} \label{areasection}

Consider a fixed convex domain \(U\subset \Omega\) bounded by a locally uniformly convex free boundary portion \(\Gamma\) and convex barrier portion of $\partial \Omega$, which, by abusing notation, we will still denote by \(\Sigma\). By denoting with $\nor_\Sigma$ the outward pointing unit normal to $\Sigma$, and $X$ the position vector, the area $|U|$, enclosed by $\Gamma$ and $\Sigma$, is given by 
\begin{align*}
	2|U| &= 2\int_{U}1dX = \int_{U}\text{div}{X}dX = \int_{\Gamma}X \cdot \nor ds + \int_{\Sigma}X \cdot \nor_{\Sigma}ds \\
	&= \int_{\underline{\theta}}^{\overline{\theta}}\frac{\sigma}{\kappa}d\theta + \int_{\Sigma}X \cdot \nor_{\Sigma}ds_{\Sigma} = \int_{\underline{\theta}}^{\overline{\theta}}\sigma(\sigma_{\theta\theta} + \sigma)d\theta + \int_{\Sigma}X \cdot \nor_{\Sigma}ds \\
	&\eqqcolon \mu(\sigma) + \mu_{\Sigma}.
\end{align*}
The first variation of \(\mu\) at \(\sigma = 1\) and along \(v \coloneqq \tfrac{d}{d\varepsilon}\big|_{\varepsilon = 0}\sigma_{\varepsilon}\) is given by 
\[
	D\mu\big|_{1}(v)=\frac{d}{d\varepsilon}\bigg|_{\varepsilon = 0}\mu(\sigma_{\varepsilon}) = v_{\theta}\big|_{\underline{\theta}}^{\overline{\theta}} + 2\int_{\underline{\theta}}^{\overline{\theta}}vd\theta\,.
\]
Therefore,
\begin{align*}
	2|U| &= \mu_{\Sigma} + \mu(1) + D\mu\big|_{1}(\sigma - 1) + O(\|\sigma - 1\|^{2}_{C^{2}}) \\
	&= \mu_{\Sigma} + \overline{\theta} - \underline{\theta} + (\sigma - 1)_{\theta}\big|_{\underline{\theta}}^{\overline{\theta}} + 2\int_{\underline{\theta}}^{\overline{\theta}}(\sigma - 1)d\theta + O(\|\sigma - 1\|^{2}_{C^{2}}) \\
	&= \mu_{\Sigma} + \overline{\theta} - \underline{\theta} + \sum_{X\in \partial\Sigma}X \cdot \nor_{\Sigma} + 2\int_{\underline{\theta}}^{\overline{\theta}}(\sigma - 1)d\theta + O(\|\sigma - 1\|^{2}_{C^{2}})\,,
\end{align*}
which yields
\begin{equation}\label{eq:unstable_node1}
	2\int_{\underline{\theta}}^{\overline{\theta}}(\sigma - 1)d\theta = 2|U| - (\overline{\theta} - \underline{\theta}) - \mu_{\Sigma} - \sum_{X\in \partial\Sigma}X \cdot \nor_{\Sigma} + O(\|\sigma - 1\|^{2}_{C^{2}}).
\end{equation}

\subsection{The center of mass functional} \label{CMfunctional}
With the same setup as in Section \ref{areasection},  we define the center of mass to be 
\[
q(U) = \int_{U}X\cdot e_1dX\,,
\]
for a domain $U$ bounded by  a locally uniformly convex free boundary portion \(\Gamma\) and a locally uniformly convex barrier portion \(\Sigma\). The center of mass is then determined by 
\begin{align*}
	2q(U) &= 2\int_{U}X\cdot e_1dX = \int_{U}(X\cdot e_1)\text{div}XdX = \int_{U}(\text{div}((X\cdot e_1)X) - X\cdot e_1)dX,
\end{align*}
which gives 
\begin{align*}
	3q(U) &= \int_{U}\text{div}((X\cdot e_1) X)dX = \int_{\Gamma}(X\cdot e_1)X \cdot \nor ds + \int_{\Sigma}(X\cdot e_1)X\cdot \nor_{\Sigma}ds \\
	&= \int_{\underline{\theta}}^{\overline{\theta}}(\sigma\cos\theta + \sigma_{\theta}\sin\theta)\sigma(\sigma_{\theta\theta} + \sigma)d\theta + \int_{\Sigma}(X\cdot e_1)X \cdot \nor_{\Sigma}ds \\
	&\eqqcolon \nu(\sigma) + \nu_{\Sigma}\,.
\end{align*}
The first variation of \(\nu\) at \(\sigma = 1\) along  \(v \coloneqq \tfrac{d}{d\varepsilon}\big|_{\varepsilon = 0}\sigma_{\varepsilon}\) is given by
\[
	D\nu\big|_{1}(v)=\frac{d}{d\varepsilon}\bigg|_{\varepsilon = 0}\nu(\sigma_{\varepsilon}) = 2v\sin\theta\big|_{\underline{\theta}}^{\overline{\theta}} + v_{\theta}\cos\theta\big|_{\underline{\theta}}^{\overline{\theta}} + \int_{\underline{\theta}}^{\overline{\theta}}v\cos\theta d\theta\,.
\]
 So preceding as with the area functional in subsection \ref{areasection} we have
\begin{align*}
	3q(U) =& \nu_{\Sigma} + \nu(1) + D\nu\big|_{1}(\sigma - 1) + O(\|\sigma - 1\|_{C^{2}}^{2}) \\
	=& \nu_{\Sigma} + [\sigma\sin\theta + \sigma_{\theta}\cos\theta + (\sigma - 1)\sin\theta]_{\underline{\theta}}^{\overline{\theta}} + \int_{\underline{\theta}}^{\overline{\theta}}(\sigma - 1)\cos\theta d\theta + O(\|\sigma - 1\|_{C^{2}}^{2}) \\ 
	=&  \int_{\Sigma}(X\cdot e_1)X \cdot \nor_{\Sigma}ds+ [X \cdot e_{2}]_{\unth}^{\ovth}+\sum_{\partial\Gamma}(1-X\cdot\nor )\nor_{\Sigma} \cdot e_{1}  \\
    &+ \int_{\underline{\theta}}^{\overline{\theta}}(\sigma-1)\cos\theta d\theta+ O(\|\sigma - 1\|_{C^{2}}^{2})\,.
\end{align*}
Here we have used the identity
\[
X\cdot e_2=(\sigma\nor +\sigma_\theta \tang)\cdot e_2=\sigma\sin\theta+\sigma_\theta\cos\theta\,,
\]
where $\tang $ is the unit tangent to the curve (counterclockwise rotation on $\nor$ by $\pi/2$). We also adopt  the notation for the difference of a function evaluated at the boundary points:  $[f]^{\ovth}_{\unth}=f(\ovth)-f(\unth)$, which typically arises from integrating a derivative.
Therefore, we obtain
\begin{equation}\label{eq:unstable_node2}
\begin{split}
	\int_{\underline{\theta}}^{\overline{\theta}}(\sigma - 1)\cos\theta\, d\theta = &3q(U) -\int_{\Sigma}(X\cdot e_1)\sigma _{\Sigma}ds - [X \cdot e_{2}]_{\unth}^{\ovth} \\
    &+ \sum_{\partial \Gamma}(1-\sigma)(\nor_{\Sigma} \cdot e_{1})+O(\|\sigma - 1\|_{C^{2}}^{2}).
    \end{split}
\end{equation}

\section{Evolution equations}\label{evolutions}

We now consider the enclosed area \(|U_{(t)}|\) and the center of mass \(q(U_{(t)})\) associated with the evolving curve \( \Gamma_t \), as defined in Sections~\ref{areasection} and~\ref{CMfunctional}. After an appropriate translation and rescaling, we expect that these quantities satisfy \( |U_{(t)}|-(\ovth-\unth) \sim 0 \) and \( q(U_{(t)}) \sim 0 \). We therefore derive their evolution equations for the rescaled flow,
\begin{equation}\label{rescaled flow}
\tilde{\Gamma}_{\ttime}
= \lambda(t)\bigl(\Gamma_t - p(t)e_1\bigr)\,,\;\;\frac{d\ttime}{dt}=\l^2.
\end{equation}

We will use the Gauss map parametrization for both the evolving curves and \(\Sigma\), which we denote by \(\gamma\) and \(\gamma_\Sigma\), respectively. Recall that the support function of \(\Sigma\), denoted by \(\sigma_\Sigma\), is defined on the interval \([\unth_\Sigma, \ovth_\Sigma]\), while the support function \(\sigma(\cdot,t)\) of \(\Gamma_t\) is defined on \([\unth, \ovth]\), with the two intervals related as described in \eqref{thetaS}. To simplify notation, we suppress the explicit dependence on \(t\) in the domains, and we adopt the same convention for other quantities throughout.

At this point, it is convenient to introduce a boundary notation that simplifies the exposition. We use the notation \(\{f\}\) to denote the sum of a function evaluated at the boundary points of \(\Gamma_t\); that is,
\[
\{f\} := \sum_{\partial \Gamma_t} f .
\]
Note that \(\partial \Gamma_t = \partial \Sigma\). Since we are using the Gauss map parametrization, the boundary points correspond to the angles \(\ovth\) and \(\unth\) when referring to functions defined on \(\Gamma_t\), and to \(\unth_\Sigma\) and \(\ovth_\Sigma\) when referring to functions defined on \(\Sigma\). Because all functions are expressed in terms of the angular parameter, the meaning of this notation will be clear from the context. For example,
\[
\{\sigma\} = \sigma(\ovth) + \sigma(\unth), 
\qquad
\{\sigma_\Sigma\} = \sigma_\Sigma(\ovth_\Sigma) + \sigma_\Sigma(\unth_\Sigma),
\]
and
\[
\{\sigma \kappa_\Sigma\}
= \sigma(\ovth)\kappa_\Sigma(\unth_\Sigma)
+ \sigma(\unth)\kappa_\Sigma(\ovth_\Sigma).
\]
Note that this notation should not be confused with the notation we will use for the \emph{difference} of a function evaluated at the boundary points, $[f]^{\ovth_\Sigma}_{\unth_\Sigma}=f(\ovth_\Sigma)-f(\unth_\Sigma)$ or $[f]^{\ovth}_{\unth} =(\ovth)-f(\unth)$, which typically arises from integrating a derivative.

We also introduce the translated position vector
\[
X_p := X - p\, e_1,
\]
and define the corresponding support functions by
\[
\sigma^p := X_p \cdot \nor, \qquad \sigma_\Sigma^p := X_p \cdot \nor_\Sigma .
\]
Finally we introduce the difference of boundary angles function 
\[
\Theta(t)=\ovth-\unth\,
\]
as well as the scale function
\[
\Lambda(t)=\frac{1}{\sqrt{2(T-t)}}\,.
\]

In terms of asymptotics, we will use the standard $O-o$ notation for quantities defined on the flow $\{\Gamma_t\}$. In particular, we will use  $O(\Lambda^k)$, $k\in \Z$, to mean that a quantity divided by $\Lambda^K$ is bounded in absolute value everywhere and for all times in $[0,T)$ and $o(\Lambda)$ to mean that moreover the bound tends to zero as $t\to T$.

\subsection{Evolution of the enclosed area functional}
We define the following quantity which is an approximation of $\int_{\unth}^{\ovth} (\sigma-1)d\theta$ for the rescaled flow, as seen by \eqref{eq:unstable_node1}.

\begin{equation}\label{U-simple}
U(t):=2\l^2(t)|U_{(t)}|-\Theta(t)-\l^2(t)\int^{\ovth_\Sigma}_{\unth_\Sigma}\frac{\sigma_\Sigma}{\kappa_\Sigma}\,d\theta-\Lambda(t)\{\sigma_\Sigma\}\,,
\end{equation}
where with $U_{(t)}$ we denote the enclosed area, as defined in section \ref{areasection}, of the rescaled flow at time $t$. Note that this does not depend on the translation by $p\cdot e_1$.

\begin{lemma} \label{U-evolution}
The evolution of $U$, defined in \eqref{U-simple}, is given by
\begin{equation}\label{dUdt}
\begin{split}
\frac{d}{dt}U(t)=
&\frac{\l_t}{\l}\left(2U(t)+2\Theta(t)+2\Lambda(t)\{\sigma_\Sigma\}\right)\\
&-\l^2(2(\ovth-\unth)-\{\sigma_\Sigma\kappa\})+\Lambda(t)\{\sigma\kappa\kappa_\Sigma\}-\{\kappa_\Sigma\kappa\}-{\Lambda^3}\{\sigma_\Sigma\}\,.
\end{split}
\end{equation}
\end{lemma}
\begin{proof}
We compute the evolution of each of the terms in the definition of $U$ separately.
Recall first \eqref{thetaev}, which gives
\begin{equation*}
\Theta'(t)= \kappa_{\Sigma}(\underline\theta_\Sigma)\kappa(\overline\theta,t) +\kappa_{\Sigma}(\overline\theta_\Sigma)\kappa(\underline\theta,t)=\{\kappa_\Sigma\kappa\}\,.
\end{equation*}

The terms involving $\Sigma$ are affected only by the evolution of the endpoints, therefore

\[
\frac{d}{dt}\int^{\ovth_\Sigma}_{\unth_\Sigma}\frac{\sigma_\Sigma}{\kappa_\Sigma}\,d\theta=-\{\sigma_\Sigma\kappa\}\,,
\]
and
\[
\frac{d}{dt}\{\sigma_\Sigma\}=-\{\sigma\kappa\kappa_\Sigma\}\,.
\]

For the enclosed area, we have
\[
\begin{split}
\frac{d}{dt}2|U_{(t)}|=&\frac{d}{dt} \int_{\unth}^{\ovth}\frac\sigma\kappa\,d\theta+\frac{d}{dt}\int_{\unth_\Sigma}^{\ovth_\Sigma}\frac{\sigma_\Sigma}{\kappa_\Sigma}\,d\theta\\
=&
\{\kappa_{\Sigma}\sigma\}-\{\sigma_{\Sigma}\kappa\}-\int_{\unth}^{\ovth} 1+\sigma(\kappa_{\theta\theta}+\kappa)\,d\theta\,.
\end{split}
\]
The last integral on the right hand side can be estimated as
\[
\begin{split}
\int_{\unth}^{\ovth} 1+\sigma(\kappa_{\theta\theta}+\kappa)
&=[\sigma\kappa_\theta]_{\unth}^{\ovth}-[\sigma_\theta\kappa]_{\unth}^{\ovth}+\int_{\unth}^{\ovth} 1+(\sigma+\sigma_{\theta\theta})\kappa\,d\theta\\
&=[\sigma\kappa_\theta]_{\unth}^{\ovth}-[\sigma_\theta\kappa]_{\unth}^{\ovth}+\int_{\unth}^{\ovth}2\,d\theta\,.
\end{split}
\]
Recalling \eqref{eq:support_pde}, we note that the boundary terms cancel and  we obtain
\begin{equation*}
\frac{d}{dt}|U_{(t)}|=-(\ovth-\unth)=-\Theta(t)\,.
\end{equation*}

Therefore, putting everything together, we have
\begin{align*}
\frac{d}{dt}U(t)=&4\l\l_t|U_{(t)}|
-2\l\l_t\int^{\ovth_\Sigma}_{\unth_\Sigma}\frac{\sigma_\Sigma}{\kappa_\Sigma}\,d\theta
-{\Lambda_t}\{\sigma_\Sigma\}
\nonumber\\
&-\l^2(2(\ovth-\unth)-\{\sigma_{\Sigma}\kappa\})+\Lambda\{\sigma\kappa\kappa_\Sigma\}-\{\kappa_{\Sigma}\kappa\}\nonumber\\
=&\frac{\l_t}{\l}\left(2U(t)+2\Theta+2\Lambda\{\sigma_\Sigma\}\right)\nonumber\\
&-{\Lambda^3}\{\sigma_\Sigma\}-\l^2(2(\ovth-\unth)-\{\sigma_\Sigma\kappa\})+\Lambda\{\sigma\kappa\kappa_\Sigma\}-\{\kappa_\Sigma\kappa\}\,.
\end{align*}
\end{proof}

\subsection{Evolution of the center of mass functional}
We define the following quantity which is an approximation of $\int_{\unth}^{\ovth} (\sigma-1)\cos\theta d\theta$ for the rescaled flow, as seen by \eqref{eq:unstable_node2}.
\begin{equation}\label{q-simple}
q(t)=
\l^3\left(3q(U^p_{(t)})
-\int^{\ovth_\Sigma}_{\unth_\Sigma}(X\cdot e_1)\frac{\sigma_\Sigma}{\kappa_\Sigma}\,d\theta - \frac{1}{\Lambda(t)^2}[\gamma\cdot e_2]^{\ovth}_{\unth}\right)\,,
\end{equation}
where $q(U^p_{(t)})$ is the center of mass, as defined in \eqref{CMfunctional}, for the rescaled flow, at time $t$.
\begin{lemma}\label{q-evolution}
The evolution of $q$, defined in \eqref{q-simple}, is given by

\begin{equation}\label{dqdt}
\begin{aligned}
\frac{d}{dt} q(t) 
&= 3\,\frac{\lambda_t}{\lambda}\, q(t) \\
&\quad + \lambda^3 \Biggl[
    - \int_{\underline{\theta}}^{\overline{\theta}} \bigl( X \cdot e_1 - p \bigr)\, d\theta
    - p_t\, |U_{(t)}| \\
&\qquad\qquad + \bigl\{ (X \cdot e_1) \sigma_\Sigma \kappa \bigr\}
    + \frac{1}{\Lambda^2} \bigl[ \kappa \sin_\Gamma \bigr]_{\underline{\theta}}^{\overline{\theta}}
    + 2 \bigl[ \gamma \cdot e_2 \bigr]_{\underline{\theta}}^{\overline{\theta}}
\Biggr].
\end{aligned}
\end{equation}

\end{lemma}
\begin{proof}
We compute the evolution of each term in the definition of $q$ separately. We first compute the evolution of the terms that involve only $\Sigma$, which only depends on the motion of the endpoints.
\[
\begin{split}
\frac{d}{dt}\int^{\ovth_\Sigma}_{\unth_\Sigma}(X\cdot e_1)\frac{\sigma_\Sigma}{\kappa_\Sigma}\,d\theta=&-\{(X \cdot e_{1})\sigma_{\Sigma}\kappa\}\,,
\end{split}
\]
and
\[
\frac{d}{dt}[\gamma\cdot e_2]^{\ovth}_{\unth}=-[\kappa\sin]^{\ovth}_{\unth}.
\]

To compute the evolution of $q(U^p_{(t)})$ we proceed as follows.
\[
\begin{split}
\frac{d}{dt}q(U^p_{(t)})&=\frac{d}{dt} \int_{\unth}^{\ovth}(X_p\cdot e_1)\frac{\sigma^p}{\kappa}\,d\theta+\frac{d}{dt}\int_{\unth_\Sigma}^{\ovth_\Sigma}(X_p\cdot e_1)\frac{\sigma^p_\Sigma}{\kappa_\Sigma}\,d\theta\\
&=\{ X_p\cdot e_1\sigma^p\kappa_\Sigma\}+\{ X_p\cdot e_1\sigma^p_\Sigma\kappa\}\\
&-\int_{\unth}^{\ovth}(X_p\cdot e_1)( 1+\sigma^p(\kappa_{\theta\theta}+\kappa))+\frac{\sigma^p}{\kappa}(\kappa\cos\theta-\kappa_\theta \sin\theta)\,d\theta\,\\
&-\int_{\unth}^{\ovth} \frac{p_t}{\kappa}(\sigma^p+X_p\cdot e_1\cos\theta)\,d\theta-\int_{\unth_\Sigma}^{\ovth_\Sigma} \frac{p_t}{\kappa_\Sigma}(\sigma^p_\Sigma+X_p\cdot e_1\cos\theta)\,d\theta\,.
\end{split}
\]
Observe that
\[
\begin{split}
\int_{\unth}^{\ovth}&(X_p\cdot e_1)( 1+\sigma^p(\kappa_{\theta\theta}+\kappa))\,d\theta
\\
=&[(X_p\cdot e_1)\sigma^p\kappa_\theta]_{\unth}^{\ovth}-[(X_p\cdot e_1)\sigma^p_\theta\kappa]_{\unth}^{\ovth}\\
&+\int_{\unth}^{\ovth}(X_p\cdot e_1) (1+(\sigma^p+\sigma^p_{\theta\theta})\kappa)\,d\theta-\int_{\unth}^{\ovth}(X_p\cdot e_1)_\theta (\sigma^p\kappa_\theta-\sigma^p_\theta\kappa)\,d\theta\\
=&[(X_p\cdot e_1)\sigma^p\kappa_\theta]_{\unth}^{\ovth}-[(X_p\cdot e_1)\sigma^p_\theta\kappa]_{\unth}^{\ovth}\\
&+\int_{\unth}^{\ovth}(X_p\cdot e_1) (1+(\sigma^p+\sigma^p_{\theta\theta})\kappa)\,d\theta+\int_{\unth}^{\ovth}\frac{\sin\theta}{\kappa} (\sigma^p\kappa_\theta-\sigma^p_\theta\kappa)\,d\theta\,,
\end{split}
\]
and since
\[
\sigma^p+\sigma^p_{\theta\theta}=\sigma+\sigma_{\theta\theta}=\frac{1}{\kappa}
\]
and 
\[
\sigma^p_\theta\sin\theta-\sigma^p\cos\theta=-X\cdot e_1+p\,,
\]
we obtain
\begin{equation}\label{q4}
\begin{split}
\frac{d}{dt}q(U^p_{(t)})&=-\int_{\unth}^{\ovth}X_p\cdot e_1\,d\theta-\frac13 p_t\int_U\text{div}(X_p+(X_p\cdot e_1) e_1)\\
&=-\int_{\unth}^{\ovth}X_p\cdot e_1\,d\theta-p_t|U_{(t)}|\,.
\end{split}
\end{equation}

Putting everything together, and noting that $\Lambda_t=\Lambda^3$ we obtain \eqref{dqdt}.
\end{proof}

\section{Dynamic stability}

The linear analysis (recall subsection \ref{sub-lenearstability}) suggests that the shrinking semicircle is dynamically stable after modding out time translations and horizontal spatial translations. To this end we will show that for the rescaling of the flow, \eqref{rescaled flow},
and for appropriately chosen functions \(\lambda\) and \(p\), the support function \(\tilde{\sigma}\) of \(\tilde{\Gamma}\) has projections onto the unstable modes that satisfy decay estimates sufficient to yield `fast' convergence to the round point. Roughly speaking, this normalization implies that, for the rescaled flow, the quantities
\[
\int_{\unth}^{\ovth} (\tsigma - 1)\, d\theta,
\qquad
\int_{\unth}^{\ovth} (\tsigma - 1)\cos\theta\, d\theta
\]
are both of order \(e^{-2\ttime}+O(|\sigma - 1|^2_{C^2})\).
We will do this in two steps, first obtaining the weaker bound \(e^{-\ttime}+O(|\sigma - 1|^2_{C^2})\), and then using this to refine our estimates and obtain the optimal bound.

\begin{lemma}\label{Ut=0}
There exists a scaling function $\l:[0, T)\to \R$ such that $U(t)=0$, for all $t\in [0, T) $, with $U$ as defined in \eqref{U-simple},  Moreover this $\l$ satisfies
\[
\l(t)-\Lambda(t)=O(1)
\]
and 
\[
\frac{\l_t}{\l^3}=1+O(\Lambda^{-1})\,.
\]
\end{lemma}
\begin{proof}
Recall the evolution of $U$, given in Lemma \ref{U-evolution}.
In order to find $\lambda$ such that $U(t)$ is identically zero,  we solve
\[
\frac{\l_t}{\l}\left(2\Theta+2\Lambda\{\sigma_\Sigma\}\right)=
\l^2(2\Theta-\{\sigma_\Sigma\kappa\})-\Lambda\{\sigma\kappa\kappa_\Sigma\}+\{\kappa_\Sigma\kappa\}+{\Lambda^3}\{\sigma_\Sigma\}\\
\]
Note that, by the convergence \eqref{1-conv}, we have that
\begin{equation}\label{ks}
\kappa(\theta)=\Lambda+E(\theta)\,\text{ and }\,\,\sigma(\theta)=\frac{1}{\Lambda}+e(\theta)\,,
\end{equation}
where the functions $E, e$ satisfy
\begin{equation}\label{Ee}
E(\theta)=o(\Lambda)\,,\,\,e(\theta)=o(\Lambda^{-1})\,.
\end{equation}
Therefore $\l$ should solve the ODE
\[
\frac{\l_t}{\l}\left(2\Theta+2\Lambda\{\sigma_\Sigma\}\right)=
\l^2(2\Theta-\Lambda\{\sigma_\Sigma\}+\{E\sigma_\Sigma\})-\{\Lambda e\kappa\kappa_\Sigma\} + \Lambda^{3}\{\sigma_{\Sigma}\}\,,
\]
which, after dividing by $\l^2$, becomes
\[
\frac{\l_t}{\l^3}=1+\frac{\{E\sigma_\Sigma\} - 3\Lambda\{\sigma_{\Sigma}\}}{2\Theta+2\Lambda\{\sigma_\Sigma\}}+\frac{1}{\l^2}\frac{\Lambda^{3}\{\sigma_{\Sigma}\} -\{\Lambda e \kappa\kappa_\Sigma\}}{2\Theta+2\Lambda\{\sigma_\Sigma\}}\,.
\]
Next, observe that $\sigma_\Sigma = O(\Lambda^{-2})$, because both 
$\sigma_{\Sigma_t}$ and its derivative $(\sigma_{\Sigma_t})_\theta$ vanish at 
$\theta = \frac{3\pi}{2}$.  
Hence, by Taylor's theorem, together with \eqref{eq:difference_in_angle} and \eqref{thetaS}, we have
\begin{equation}\label{Taylor}
\sigma_{\Sigma_t}(\theta)
= \frac{1}{2} (\theta - \tfrac{3\pi}{2})^2 \, \kappa_{\Sigma_t}^{-1}\!\bigl(\tfrac{3\pi}{2}\bigr)
  + O\bigl((\theta - \tfrac{3\pi}{2})^3\bigr)
= O(\Lambda^{-2})\,.
\end{equation}
Using \eqref{Ee} and \eqref{Taylor},  the ODE can be  simplified as 
\[
 2\frac{\l_t}{\l^3}= 2+O\left(\frac{1}{\Lambda}\right)+\frac{1}{\l^2}O(\Lambda)\,.
\]
Hence, any solution with $\l(t)\to \infty$ as $t\to T$, satisfies
\[
\frac{1}{\l^2}=\frac{1}{\Lambda^2}+O(\Lambda^{-3})\,.
\]
To see this, note that there exists a constant $C$ such that,
\[
 2-\frac{C}{\Lambda}-\frac{C\Lambda}{\l^2}\le 2\frac{\l_t}{\l^3}\le 2+\frac{C}{\Lambda}+\frac{C\Lambda}{\l^2}\,.
\]
This implies
\[
\left(e^{\frac{C }{\Lambda}}\l^{-2}\right)_t\ge e^{\frac{C }{\Lambda}}\left(2-\frac{C}{\Lambda}\right)\,\text{ and }\,\left(e^{-\frac{C  }{\Lambda}}\l^{-2}\right)_t\le e^{-\frac{C  }{\Lambda}}\left(2+\frac{C}{\Lambda}\right)\,.
\]
Integrating these inequalities from $t$ to $T$ then gives the stated estimate.
Rearranging yields the required estimates on $\lambda$.

Finally, we note that with $\lambda$ as above, $U(t)=0$. Indeed, this choice of $\lambda$, together with \eqref{dUdt}, implies that
\[
U'(t)=\frac{2\l_t}{\l} U(t)\,.
\]
Therefore 
\[
\frac{d}{dt}\left(\l^{-2} U(t)\right)=0\,,
\]
and since $\lim_{t\to T}U(t)=0$, the result follows.
\end{proof}

\begin{lemma}\label{qt=0}
Consider $q$ be as defined in \eqref{q-simple}, where $\l$ is such that $\l^{-1}$ is bounded. Then, 
there exists a translation function $p:[0, T)\to \R$, independent of $\l$, such that $q(t)=0$ for all $t\in [0,T)$. Moreover, $p$ satisfies
\[
p=O(\Lambda^{-1})\,\,\text{ and }\,\,
p_t=O\left(\Lambda\right)\,.
\]
\end{lemma}
\begin{proof}
Recalling the evolution equation of $q$, \eqref{dqdt}, we 
choose $p$ so that
\[
\begin{split}
p_t|U_{(t)}|-p\Theta=-\int_{\unth}^{\ovth} X\cdot e_1\,d\theta
+\{(X\cdot e_1)\sigma_\Sigma\kappa\}
+\frac{1}{\Lambda^2}[\kappa\sin]^{\ovth}_{\unth}+2[\gamma\cdot e_2]^{\ovth}_{\unth}\,.
\end{split}
\]
Notice that this can be written as 
\begin{equation}\label{pt}
(p|U_{(t)}|)_t=-\Delta(t)+\delta(t)\,,
\end{equation}
where $\Delta(t)=-\int_{\unth}^{\ovth} X\cdot e_1\,d\theta$ and $\delta(t)$ is the rest of the right hand side,
and observe that these functions satisfy
\begin{equation}\label{dD}
\Delta(t)=o(\Lambda^{-1}) \text{ and }\delta(t)=O(\Lambda^{-2})\,.
\end{equation}
Indeed, this can be easily verified by expressing the position in terms of the support function,
\[
X=\sigma(\cos\theta, \sin\theta)+\sigma_\theta (-\sin\theta, \cos\theta)\,,
\]
and using the asymptotics \eqref{ks}.  The estimate for $\delta(t)$ is then straightforward after recalling \eqref{Taylor} and \eqref{eq:support_pde}, whereas for $\Delta (t)$ is evident after integration by parts
\[
\Delta (t)=[-\sigma\sin\theta]^{\ovth}_{\unth}+2\int_{\unth}^{\ovth}\sigma\cos\theta d\theta=[(\Lambda^{-1} -e(\theta))\sin\theta]^{\ovth}_{\unth}+2\int_{\unth}^{\ovth}e(\theta)\cos\theta d\theta\,.
\]
Therefore, integrating \eqref{pt} from $t$ to $T$, we obtain
\[
p|U_{(t)}|= O(\Lambda^{-3})\,.
\]
Now, recall the expression for $|U_{(t)}|$ from Section~\ref{areasection}.  
Together with \eqref{ks} and \eqref{Taylor}, this implies that $
|U_{(t)}| = O(\Lambda^{-2})$
and therefore we obtain
\[
p = O(\Lambda^{-1})\,.
\]
Note also that \eqref{pt}, together with the estimates \eqref{dD}, give the required $p_t$ estimate.

Finally, with this choice of $p$, and by \eqref{dqdt}, the evolution for $q$ becomes
\[
\left(\frac{q}{\l^3}\right)_t=0
\]
and since $q(t)\to 0$ and $t\to T$, and $\l^{-1}$ is assumed bounded, integrating implies that $q$ has to be constantly zero.
\end{proof}

We can now obtain estimates on the unstable modes. As it will be convenient to express them in terms of the new time variable $\ttime$, we first note that 
\begin{equation}\label{NTV}
e^{-\ttime}=O(\lambda^{-1})=O(\Lambda^{-1})\,.
\end{equation}
To verify this, recall that, as in defined \eqref{rescaled flow}, the new time variable satisfies
\[
	\begin{cases}
		\frac{d\ttime}{dt} = \lambda^{2} \\
		\ttime(0) = -\frac{1}{2}\log(2T).
	\end{cases}
\]	
Using the inequalities for \(\lambda\) in Lemma \ref{Ut=0}, we have 
\[
\frac{d\ttime}{dt}=\l^2=\Lambda^2+ O(\Lambda)\,,
\]
integration of which yields the estimate
$
e^{\ttime}=O(\Lambda)$ and therefore \eqref{NTV}.

\begin{lemma}\label{lowmodes} With $\l$ and $p$ as in Lemmas \ref{Ut=0} and \ref{qt=0}, the rescaled flow \eqref{rescaled flow} satisfies 
\[
\int_{\unth}^{\ovth}(\tsigma-1)=O(e^{-\ttime})+O(\|\tsigma-1\|_{C^2}^2)
\]
and 
\[
\int_{\unth}^{\ovth}(\tsigma-1)\cos\theta=O(e^{-\ttime})+O(\|\tsigma-1\|_{C^2}^2)\,.
\]
\end{lemma}
\begin{proof}
The expression \eqref{eq:unstable_node1}, and the estimates on $\lambda$ given in Lemma \ref{Ut=0}, give
\[
\begin{split}
2\int_{\unth}^{\ovth}(\tsigma-1)&=2\l^2(t)|U_{(t)}|-(\overline\theta-\underline\theta)-\l^2(t)\int^{\ovth_\Sigma}_{\unth_\Sigma}\frac{\sigma^p_\Sigma}{\kappa_\Sigma}\,d\theta-\l(t)\{\sigma^p_\Sigma\}+O(\|\tsigma - 1\|_{C^{2}}^{2})\\
&=U(t)+O(\Lambda^{-1})+O(\|\tsigma - 1\|_{C^{2}}^{2})\,.
\end{split}
\]
The expression \eqref{eq:unstable_node2}, Lemma \ref{Ut=0}, and the estmates on $p$ given in Lemma \ref{qt=0}, give
\[
\begin{split}
2\int_{\unth}^{\ovth}(\tsigma-1)\cos\theta=&
3\l^3(t)q(U^p_{(t)})
-\l^3(t)\int^{\ovth_\Sigma}_{\unth_\Sigma}(X_p\cdot e_1)\frac{\sigma^p_\Sigma}{\kappa_\Sigma}\,d\theta - \l(t)[\gamma\cdot e_2]^{\ovth}_{\unth}\\
&+\sum_{\partial \Gamma}(1-\sigma)(\nor_{\Sigma} \cdot e_{1})+O(\|\tsigma - 1\|_{C^{2}}^{2})\\
=&q(t)+O(\Lambda^{-1})+O(\|\tsigma - 1\|_{C^{2}}^{2})\,.
\end{split}
\]
The result now follows by  Lemmas \ref{Ut=0} and \ref{qt=0}, and \eqref{NTV}.
\end{proof}

\section{\texorpdfstring{$L^2$ decay}{L2  decay}}\label{L2decay}

We consider now the normalised flow as defined in \eqref{rescaled flow} with $\lambda$ and $p$ as defined in Lemmas \ref{Ut=0}, \ref{qt=0}. Then, parametrizing by the Gauss map, and skipping the tildes here, the support function satisfies:
\begin{equation}\label{F-speed}
	\partial_{t}\sigma = L \, \sigma + B\,\sigma\cos\theta - (\sigma_{\theta\theta} + \sigma)^{-1} \eqqcolon F(\sigma)\,,
\end{equation}
where 
\begin{equation}\label{LB}
L:=\frac{\l_t}{\l}\,,\,\, B:=-\lambda p_t\,,
\end{equation}
and note also that $t$ is the new time variable defined in $t\in (-\frac12\log(2T), +\infty)$. Note that, by Lemmas \ref{Ut=0}, \ref{qt=0} and \eqref{NTV},
\begin{equation}\label{LBest}
|L-1|\le Ce^{-t}\,,\,\,|B|\le Ce^{-t}\,.
\end{equation}
Moreover,
the angle domains for the support function are not altered, however with the new time variable their 
\[
\frac{d}{dt}\Theta =O(\l^{-1})= O(e^{-t})
\]
We will show that $\sigma-1$ converges to $0$ exponentially fast.
\begin{theorem}\label{thm:stilde!}
For any $k\in \N$ and $\alpha<1$, there exists $C(k, \alpha)$ such that the rescaled flow, \ref{rescaled flow}, with $\lambda$ and $p$ as defined in Lemmas \ref{Ut=0}, \ref{qt=0}, satisfies 
\begin{equation*}
\|\sigma-1\|_{C^k}\le C(k, \alpha) e^{- \alpha t}\,, \forall t\in  (-\tfrac12\log(2T), +\infty)\,.
\end{equation*}
\end{theorem}
\begin{remark} To prove Theorem \ref{thm:stilde!} we first prove  decay of the $L^2$ norm, which is a bit stronger, 
\[
\|\sigma-1\|_{L^2}\le C e^{-2t}\,,\,\,\forall t\in  (-\tfrac12\log(2T), +\infty)\,.
\]
On higher derivatives the decay factor has to drop because we use interpolation to achieve the estimates.
\end{remark}
\begin{proof}
We will show decay of $\|\sigma-1\|_{L^2(\unth, \ovth)}$ by computing and estimating its evolution. The higher order norm estimates will then follow via interpolation.

We consider $L, B$, as defined in \eqref{LB}, \emph{solely} as  functions of $t$, so that the linearisation of the speed \(F\), as in \eqref{F-speed},  about \(\sigma = 1\) and along  \(v \coloneqq \tfrac{d}{d\varepsilon}\big|_{\varepsilon = 0}\sigma_{\varepsilon}\) is given by 
\[
	DF\big|_{1}v = v_{\theta\theta} + (L + 1)v + B v\cos\theta\,.
\]
Therefore we obtain
\begin{equation}\label{L2-evolution}
\begin{split}
	\frac{d}{dt}\int_{\underline{\theta}}^{\overline{\theta}}(\sigma - 1)^{2}d\theta &\leq Ce^{-t}\|\sigma - 1\|_{C^{0}}^{2} \\
    &+ 2\int_{\underline{\theta}}^{\overline{\theta}}\Big[(\sigma - 1)\Big(F(1) + DF\big|_{1}(\sigma - 1)\Big) + C|\sigma - 1|\|\sigma - 1\|_{C^{2}}^{2}\Big]d\theta \,.
\end{split}
\end{equation}
Since 
\[
	F(1) = L-1+B\cos\theta, 
\]
we have, by Lemma \ref{lowmodes} and \eqref{LBest}, and applying the Cauchy-Schwarz inequality,
\[
2\int_{\unth}^{\ovth}(\sigma-1) F(1)\,d\theta\le  C\left(e^{-2t}+\|\sigma-1\|_{C^2}^4\right)\,.
\]
The last term on the right hand side of \eqref{L2-evolution} is estimated, using H\"older's inequality and Cauchy-Schwarz, as
\[
\int_{\underline{\theta}}^{\overline{\theta}} |\sigma - 1|\|\sigma - 1\|_{C^{2}}^{2}d\theta \le C 
\|\sigma-1\|_{C^2}^2\|\sigma-1\|_{L^2}\le \frac{C}{\epsilon}
\|\sigma-1\|_{C^2}^4+\epsilon \|\sigma-1\|^2_{L^2}\,,
\]
for any $\epsilon>0$.

Finally,  for the gradient term we have 
\[
2\int_{\underline{\theta}}^{\overline{\theta}}(\sigma - 1)DF\big|_{1}(\sigma - 1) d\theta =
2 \int_{\underline{\theta}}^{\overline{\theta}}(\sigma - 1)\Big((\sigma - 1)_{\theta\theta} + 2(\sigma - 1)\Big)d\theta+O(e^{-2t})+O(\|\sigma-1\|_{C^2}^4)\,.
\]

In order to estimate the first term on the right hand side
we construct a \(C^{1}\) (piecewise \(C^{2}\)) extension  of \(\sigma-1\) in the interval \([0,\pi]\) with Neumann boundary condition at the endpoints. This construction is performed independently for each time slice; hence, we suppress the  $t$-dependence in the notation. We define the extension by setting
\begin{equation}\label{extension}
	\overline \sigma(\theta)-1 = \begin{cases}
		\underline{a}\theta^{2} + \underline{b}\theta +\underline{c} & \theta \in [0,\underline{\theta}] \\
		\overline{a}\theta^{2} + \overline{b}\theta + \overline{c} & \theta \in [\overline{\theta},\pi],
	\end{cases}
\end{equation}
where
\[
	\underline{a} = \frac{\sigma_{\theta}(\underline{\theta})}{2\underline{\theta}}, \hspace{0.5cm} \underline{b} = 0, \hspace{0.5cm} \underline{c}= \sigma(\underline{\theta})-1 - \underline{a}\:\underline{\theta}^{2}, 
\]
and
\[
	\overline{a} = - \frac{\sigma_{\theta}(\overline{\theta})}{2(\pi - \overline{\theta})}, \hspace{0.5cm} \overline{b} = -2\pi\overline{a}, \hspace{0.5cm} \overline{c} = \sigma(\overline{\theta})-1 - \overline{a}\overline{\theta}^{2} - \overline{b}\:\overline{\theta}.
\]

 Recalling \eqref{eq:support_pde} and \eqref{Taylor}, we note that $\underline a$, $\overline a$ are bounded, $\osigma$ is decreasing in $[0, \unth]$ and increasing in $[\ovth, \pi]$, and also
 \begin{equation}\label{ext-diff}
\|\sigma-1\|_{C^0}\le \|\overline\sigma-1\|_{C^0}\le \|\sigma-1\|_{C^0}+ Ce^{-2t}\,.
 \end{equation}

The gradient term we want to estimate, can now be written as  
\[
\int_{\underline{\theta}}^{\overline{\theta}}(\sigma - 1)\Big((\sigma - 1)_{\theta\theta} + 2(\sigma - 1)\Big)d\theta=\int_{0}^{\pi}(\osigma - 1)\Big((\osigma - 1)_{\theta\theta} + 2(\osigma - 1)\Big)d\theta+ Err(\osigma, \sigma)\,,
\]
where the error term, $Err(\osigma, \sigma)$, can be bounded  as follows.
\begin{equation}\label{graderror}
\begin{split}
Err(\osigma, \sigma)=&-\int_{[0,\pi]\setminus[\underline\theta,\overline\theta]}(\bar\sigma - 1)\Big((\bar\sigma - 1)_{\theta\theta} + 2(\bar\sigma - 1)\Big)d\theta\\
&\le -\int_{[0,\pi]\setminus[\underline\theta,\overline\theta]}(\bar\sigma - 1)(\bar\sigma - 1)_{\theta\theta} d\theta
\le C e^{-t}\|\osigma-1\|_{L^2}\le \frac{C}{\epsilon} e^{-2t}+\epsilon\|\sigma-1\|^2_{L^2} \,,
\end{split}
\end{equation}
for any $\epsilon>0$.

We finally estimate the term 
\[
\int_{0}^{\pi}(\osigma - 1)\Big((\osigma - 1)_{\theta\theta} + 2(\osigma - 1)\Big)d\theta
\]
by expressing   $\osigma-1$   as
$\osigma-1=\sum_j\sigma_j\phi_j$ in $L^2([0, \pi])$, where $\phi_j=\cos (j\theta)$ and  $\sigma_j=\langle \osigma-1,\phi_j\rangle$.
Since $\osigma$ it is a $C^1$, and piecewise $C^2$ function, with Neumann boundary data, we have
\[
\int_0^\pi\phi(\osigma_{\theta\theta}+2\osigma)=\int_0^\pi\osigma(\phi_{\theta\theta}+2\phi)\,,
\]
for any $\phi$ also satisfying Neumann boundary data. 
The $L^2$ convergence, implies weak $L^2$ convergence and therefore
\begin{equation}\label{IBP}
\begin{split}
\int_0^\pi(\osigma-1)&((\osigma-1)_{\theta\theta}+2(\osigma-1))\\&=
\int_0^\pi\sum_j\sigma_j \phi_j((\osigma-1)_{\theta\theta}+2(\osigma-1))\\
&=\lim_N\sum^N_j\sigma_j \int_0^\pi(\osigma-1)((\phi_j)_{\theta\theta}+2\phi_j)\\
&=\lim_N\sum^N_j(2-j^2)\sigma_j \int_0^\pi(\osigma-1)\phi_j=\lim_N\sum^N_j(2-j^2)\sigma_j^2 \\
&\le 4(\sigma_0^2+\sigma_1^2)-2\lim_N\sum_{j=0}^N\sigma_j^2=4(\sigma_0^2+\sigma_1^2)-2\|\osigma-1\|_{L^2}^2\,.
\end{split}
\end{equation}
The two unstable nodes can be controled, using \eqref{ext-diff} and Lemma \ref{lowmodes},
\begin{equation}\label{s0s1}
\begin{split}
	\sigma_{0}^{2} + \sigma_{1}^{2} &\leq C\left[\left(\int_{\underline{\theta}}^{\overline{\theta}}(\sigma - 1)d\theta\right)^{2} + \left(\int_{\underline{\theta}}^{\overline{\theta}}(\sigma - 1)\cos\theta d\theta\right)^{2} + e^{-4t} + \|\sigma - 1\|_{C^{2}}^{4}\right]\\
    &\le C\left(e^{-2t} + \|\sigma - 1\|_{C^{2}}^{4}\right)
    \end{split}\,. 
\end{equation}
Therefore, replacing all the above estimates in \eqref{L2-evolution}, we obtain, for any $\epsilon>0$,
\[
\begin{split}
	\deriv{t}{\|\sigma - 1\|_{L^{2}}^{2}} \leq \tfrac{C}{\epsilon} (e^{-2t}+ \|\sigma - 1\|_{C^{2}}^{4}) - (4-\epsilon)\|\sigma - 1\|_{L^{2}}^{2} \,.
    \end{split}
\]

The Gagliardo-Nirenberg interpolation inequality in bounded domains yields
\[
\|\sigma - 1\|_{C^2}\le C \|\sigma - 1\|^\frac{1}{3}_{C^{7
}}\|\sigma - 1\|_{L^2}^{\frac{2}{3}}+C\|\sigma - 1\|_{L^2}\,.
\]

Since, by \eqref{1-conv}, $\|\sigma-1\|_{C^k}=o(1)$, for all $k\in \N$,  we can absorb the  $\|\sigma - 1\|_{C^{2}}^{4}$  term and obtain that for any $\epsilon>0$, there exists $t_\epsilon\ge -\frac12\log(2T)$ such that 
\[
	\deriv{t}{\|\sigma - 1\|_{L^{2}}^{2}} \leq Ce^{-2t}  - (4-\epsilon)\|\sigma - 1\|_{L^{2}}^{2}\,,\,\,\forall t\ge t_\epsilon\,. 
\]
This then yields that 
\begin{equation}\label{stilde}
    \|\sigma-1\|_{L^2}\le C e^{-t}\,,\,\,     \forall t\ge \tfrac12\log(2T),
\end{equation}
with $C$ depending on $\|\sigma(\cdot, \tfrac12\log(2T))-1\|_{C^0}$ and $T$, since we consider the estimate in the whole time interval.
Interpolation, via for instance the Gagliardo-Nirenberg inequality as above, yields the required estimates.
\end{proof}

We can now use this to get estimates for the original (unscaled) solution.

\begin{theorem}\label{first decay}
The support function of the free boundary curve shortening flow solution $\{\Gamma_t\}_{t\in [0, T)}$ satisfies, for any $\alpha<1$, 
\[
\sigma=\sqrt{2(T-t)}+O((T-t)^{\frac{1+\alpha}{2}})
\]
and 
\[
\kappa=\frac{1}{\sqrt{2(T-t)}}+O((T-t)^{\frac{\alpha-1}{2}})\,.
\]
\end{theorem}
\begin{proof}
By Theorem \ref{thm:stilde!}, and \eqref{NTV}, for any $\alpha <1$,
\[
|\l(\sigma-p)-1|=|\tsigma-1|=O(\Lambda^{-\alpha})
\]
and
\[
|\l^{-1}\kappa-1|=|\tkappa-1|=O(\Lambda^{-\alpha})\,,
\]
where here we reintroduce tildes for the rescaled flow \eqref{rescaled flow}.
The estimate for $\kappa$ is then immediate recalling the estimates for $\l$ from Lemma \ref{Ut=0}. To prove the estimate for the support function, 
we will first show that we can improve the estimates on $p$ given in 
in Lemma \ref{qt=0}. Recall that 
\[
p_t=-\frac{1}{|U_{(t)}|}\int_{\unth}^{\ovth} X_p\cdot e_1\,d\theta+\frac{\delta(t)}{|U_{(t)}|}=-\frac{1}{|U_{(t)}|\l}\int_{\unth}^{\ovth} \tilde X\cdot e_1\,d\theta+\frac{\delta(t)}{|U_{(t)}|}\,,
\]
where $\delta(t)$ is as defined in \eqref{dD}.
Since $\tilde X\cdot e_1=\tsigma\cos\theta-\tsigma_\theta\sin\theta$, we have 
\[
\begin{split}
\int_{\unth}^{\ovth} \tilde X\cdot e_1&=\int_{\unth}^{\ovth} \tsigma\cos\theta-\tsigma_\theta\sin\theta\,d\theta=2\int_{\unth}^{\ovth} \tsigma\cos\theta\,d\theta-[\tsigma\sin\theta]_{\unth}^{\ovth}\\
&=2\int_{\unth}^{\ovth} (\tsigma-1)\cos\theta\,d\theta+[(2-\tsigma)\sin\theta]_{\unth}^{\ovth}\,.
\end{split}
\]
Therefore, by Lemma \ref{lowmodes}, the above quantity is of the order $O(\Lambda^{-1})$ and, together with \eqref{dD}, we obtain the improved estimate.
\[
p_t= O(1)\,.
\]
Since $p(T)=0$, integrating from $t$ to $T$ we find
\begin{equation}\label{newp}
p(t)= O(\Lambda^{-2})\,.
\end{equation}
We can now get the required estimate by scaling back and recalling the estimates for $\l$ from Lemma \ref{Ut=0}
\[
\sigma=\frac{1}{\Lambda}+ O(\Lambda^{-1-\alpha})\,.
\]
\end{proof}

Note that in the proof of Theorem \ref{first decay}, we showed that after we know a certain decay on $\sigma$ (Lemma \ref{lowmodes} is sufficient), we can go back to Lemma \ref{qt=0} and improve on the estimate for $p$. We claim that we can also improve the estimate for $\lambda$ in Lemma \ref{Ut=0}. To do this, note that the quantities $E$ and $e$, as defined in \ref{ks}, due to Theorem \ref{first decay}, satisfy the improved estimates
\begin{equation}\label{Ee2}
E(\theta)=O(\Lambda^{1-\alpha})\,,\,\,e(\theta)=O(\Lambda^{-1+\alpha})\,,\,\, \forall \alpha <1\,.
\end{equation}
Therefore, following the proof of Lemma \ref{Ut=0}, the ODE that $\lambda$ solves can be better estimated as 
\[
 2\frac{\l_t}{\l^3}= 2+O\left(\frac{1}{\Lambda^{1+\alpha}}\right)+\frac{1}{\l^2}O(\Lambda^{1-\alpha})\,.
\]
Hence, a solution with $\l(t)\to \infty$ as $t\to T$, satisfies
\[
\frac{1}{\l^2}=\frac{1}{\Lambda^2}+O(\Lambda^{-3-\alpha})\,,
\]
which by rearranging gives
\begin{equation}\label{newl}
\l=\Lambda+O(\Lambda^{-\alpha})\,,\,\,\forall \alpha<1\,.
\end{equation}

We can now improve Lemma \ref{lowmodes} to the following.
\begin{Corollary}\label{lowmodes2} With $\l$ and $p$ as in Lemmas \ref{Ut=0} and \ref{qt=0}, the rescaled flow \eqref{rescaled flow} satisfies, for any $\alpha<1$,
\[
\int_{\unth}^{\ovth}(\tsigma-1)=O(e^{-(1+\alpha)\ttime})+O(\|\tsigma-1\|_{C^2}^2)
\]
and 
\[
\int_{\unth}^{\ovth}(\tsigma-1)\cos\theta=O(e^{-(1+\alpha)\ttime})+O(\|\tsigma-1\|_{C^2}^2)\,.
\]
\end{Corollary}
\begin{proof}
Note that by \eqref{eq:unstable_node1} and \eqref{eq:unstable_node2}, together with the new estimates on $p$ and $\lambda$ given in \eqref{newp}, \eqref{newl}, we have, for any $\alpha<1$, 
\[
\begin{split}
2\int_{\unth}^{\ovth}(\tsigma-1)&=2\l^2(t)|U_{(t)}|-(\overline\theta-\underline\theta)-\l^2(t)\int^{\ovth_\Sigma}_{\unth_\Sigma}\frac{\sigma^p_\Sigma}{\kappa_\Sigma}\,d\theta-\l(t)\{\sigma^p_\Sigma\}+O(\|\tsigma - 1\|_{C^{2}}^{2})\\
&=U(t)+O(\Lambda^{-1-\alpha})+O(\|\tsigma - 1\|_{C^{2}}^{2})\,,
\end{split}
\]
and 

\[
\begin{split}
2\int_{\unth}^{\ovth}(\tsigma-1)\cos\theta=&
3\l^3(t)q(U^p_{(t)})
-\l^3(t)\int^{\ovth_\Sigma}_{\unth_\Sigma}(X_p\cdot e_1)\frac{\sigma^p_\Sigma}{\kappa_\Sigma}\,d\theta - \l(t)[\gamma\cdot e_2]^{\ovth}_{\unth}\\
&+\sum_{\partial \Gamma}(1-\sigma)(\nor_{\Sigma} \cdot e_{1})+O(\|\tsigma - 1\|_{C^{2}}^{2})\\
=&q(t)+O(\Lambda^{-1-\alpha})+O(\|\tsigma - 1\|_{C^{2}}^{2})\,.
\end{split}
\]
The result now follows by  Lemmas \ref{Ut=0} and \ref{qt=0}, and \eqref{NTV}.
\end{proof}

\begin{remark} Even though we can improve on the decay for the unstable modes, we cannot expect any better decay for the full $L^2$ norm, which is due to the boundary values. In particular, on the boundary we have $|\tsigma_\theta|=\tsigma_\Sigma$ which is of the order $e^{-t}$. Recalling \eqref{thetaev} this cannot be improved.
\end{remark}

\subsection*{Acknowledgements}

TB and NB were supported by the grant NSF 2405007. ML was supported by the ARC Discovery Project grant DP250103952.

\printbibliography

\end{document}